\documentclass[11pt]{article}

\usepackage{amsmath}
\usepackage{amssymb}
\usepackage{amscd}
\usepackage{amsthm}
\usepackage{indentfirst}

\newcommand{\Z}{\ensuremath{\mathbb{Z}}}
\newcommand{\Q}{\ensuremath{\mathbb{Q}}}

\theoremstyle{plain}
\newtheorem{thm}{Theorem}[section]
\newtheorem{lem}[thm]{Lemma}
\newtheorem{cor}[thm]{Corollary}
\newtheorem{prop}[thm]{Proposition}

\newtheorem{conjecture}[thm]{Conjecture}
\theoremstyle{definition}

\newtheorem{rmq}[thm]{Remark}

\DeclareMathOperator{\Spec}{Spec}
\DeclareMathOperator{\Pic}{Pic}
\DeclareMathOperator{\Jac}{Jac}
\DeclareMathOperator{\Cl}{Cl}
\DeclareMathOperator{\Gal}{Gal}
\DeclareMathOperator{\divisor}{div}

\newcommand{\J}{\mathcal{J}}
\newcommand{\C}{\mathcal{C}}
\newcommand{\fppf}{\text{\rm fppf}}
\newcommand{\tors}{\text{\rm tors}}

\begin{document}
\bibliographystyle{amsalpha}

\title{Galois module structure and Jacobians of Fermat curves}

\author{Philippe Cassou-Nogu\`es \and Jean Gillibert \and Arnaud Jehanne}

\date{April 16, 2014}

\maketitle

\begin{abstract}
The class-invariant homomorphism allows one to measure the Galois module structure of extensions obtained by dividing points on abelian varieties. In this paper, we consider the case when the abelian variety is the Jacobian of a Fermat curve. We give examples of torsion points whose associated Galois structure is trivial, as well as points of infinite order whose associated Galois structure is non-trivial.
\end{abstract}

%%%%%%%%%%%%%%%%%%%%%%%%%%%%%%%%%%%%%%%%%%%%%%%%%%%%%%%%%%%%%%%%%%%%%%%%%%%%

\section{Introduction}

In \cite{taylor88}, M. J. Taylor initiated the study of certain Galois module structures attached to $(A, \lambda)$ where $A$ denotes  an abelian variety defined over a number field and $\lambda$ is an endomorphism of $A$. This work has played a key role in new developments in the subject. On the one hand it has provided a new situation where various generalizations  of the work of Fr\"ohlich \cite{F} could be considered. On the other hand it has been a convincing path of the theory to a geometric  set up, which has proved to be fruitful since (see \cite{erez98} or \cite{CEPT} for instance). While these structures are well understood when $A$ is an elliptic curve, almost nothing is known in higher dimension. This note is a first attempt to fill this gap. More precisely, our goal is to study the Galois module structures attached to $(A, \lambda)$ when $A$ is the Jacobian of a Fermat curve and $\lambda$ is an algebraic integer of the field of complex multiplication of $A$.

%%%%%%%%%%%%%%%%%%%%%%%%%%%%%%%%%%%%%%%%%%%%%
\subsection{Galois modules and abelian varieties}

We now fix some notations and briefly recall the situation considered in \cite{taylor88}, that is, the construction of the so-called \emph{class-invariant homomorphism}.

We let $\overline{\Q}$ denote the algebraic closure of $\Q$ in $\mathbb{C}$. If $L$ is a number field contained in $\overline{\Q}$ we denote by $\mathcal{O}_L$ its ring of algebraic integers. If $S$ is a finite set of finite places of $L$, we denote by $\mathcal{O}_{L,S}$ the ring of $S$-integers in $L$.

We consider an abelian variety $A$ defined over a number field $L$, and we let $\lambda:A\to A$ be an endomorphism of $A$. We denote by $G$ the kernel of $\lambda:A(\overline{\Q})\to A(\overline{\Q})$. The group $G$, endowed with the action of $\Gal(\overline{\Q}/L)$, can be seen as a group scheme over $L$; this is just the kernel of $\lambda$. We assume for simplicity that the elements of $G$ are rational over $L$. In this case, $G$ is a constant group scheme and thus is the spectrum of the Hopf algebra $\mathrm{Map}(G,L)$. Its Cartier dual $G^D$ is the spectrum of the group algebra $L[G]$.

Let $S$ be a finite set of places of $L$ containing all places of bad reduction of $A$, and let $\mathcal{A}$ be the N\'eron model of $A$ over the ring $\mathcal{O}_{L,S}$, which is an abelian scheme. The endomorphism $\lambda$ extends to an endomorphism of $\mathcal{A}$, that we still denote by $\lambda$. We denote by $\mathcal{G}$ the kernel of this endomorphism and by $\mathcal{G}^D$ its Cartier dual. Then $\mathcal{G}$ and $\mathcal{G}^D$ are finite and locally free over $\Spec(\mathcal{O}_{L,S})$, and we have an exact sequence for the fppf topology on $\Spec(\mathcal{O}_{L,S})$
\begin{equation}
\label{suite1}
\begin{CD}
0 @>>> \mathcal{G} @>>> \mathcal{A} @>\lambda>> \mathcal{A} @>>> 0.\\
\end{CD}
\end{equation}
Let
$$
\delta:A(L)=\mathcal{A}(\mathcal{O}_{L,S})\longrightarrow H^1(\mathcal{O}_{L,S},\mathcal{G})
$$
be the coboundary map deduced from the exact sequence \eqref{suite1}. Here, we write $H^1(\mathcal{O}_{L,S},-)$ as a short hand for $H^1_{\fppf}(\Spec(\mathcal{O}_{L,S}),-)$, computed with respect to the fppf topology on $\Spec(\mathcal{O}_{L,S})$. We observe that $A(L)=\mathcal{A}(\mathcal{O}_{L,S})$ by the universal property of the N\'eron model.

In 1971, Waterhouse \cite{w71} has defined a morphism
$$
\pi:H^1(\mathcal{O}_{L,S},\mathcal{G})\longrightarrow \Pic(\mathcal{G}^D)
$$
which measures the Galois structure of $\mathcal{G}$-torsors. The \emph{class-invariant homomorphism} $\Psi_{L,S}$ is defined by the composition of the maps
$$
\begin{CD}
\Psi_{L,S}: A(L) @>\delta>> H^1(\mathcal{O}_{L,S},\mathcal{G}) @>\pi>> \Pic(\mathcal{G}^D) \\
\end{CD}
$$

When $S$ is empty (which is possible only if $A$ has everywhere good reduction), we denote for simplicity $\Psi_{L,\emptyset}$ by $\Psi_L$.

Let us describe this construction in algebraic terms. Let $\mathfrak{B}$ (resp. $\mathfrak{A}$) be the $\mathcal{O}_{L,S}$-Hopf algebra of the group scheme $G$ (resp. $G^D$). Then $\mathfrak{B}$ (resp. $\mathfrak{A}$) is a Hopf order in $L[G]$ (resp. $\mathrm{Map}(G,L)$). Given $P\in A(L)$, we consider the $G$-set
$$
G_P:=\{Q\in A(\overline{\Q}) ~|~ \lambda Q=P\}.
$$
We denote by $L(\lambda^{-1}P)$ the smallest extension of $L$ over which the elements of $G_P$ are defined.
Then $\mathrm{Map}_{\Gal(\overline{\Q}/L)}(G_P,\overline{\Q})$ is a $G$-Galois algebra over $L$, product of copies of $L(\lambda^{-1}P)$, and the algebra of the torsor $\delta(P)$ is an order in this algebra, that we denote by $\mathfrak{C}_P$. This order is a Hopf-Galois object over the algebra $\mathfrak{B}$, hence a locally-free rank one $\mathfrak{A}$-module. The class-invariant homomorphism $\Psi_{L,S}$ can be described as the map
\begin{align*}
A(L) &\longrightarrow \Cl(\mathfrak{A}) \\
P &\longmapsto (\mathfrak{C}_P)-(\mathfrak{B})
\end{align*}
where $\Pic(G^D)$ is identified with $\Cl(\mathfrak{A})$, the locally free class group of $\mathfrak{A}$. Hence, $\Psi_{L,S}$ measures the Galois module structure of $\mathfrak{C}_P$.

In \cite{taylor88}, Taylor made the following conjecture:

\begin{conjecture}
\label{martin_conj}
Let $A/L$ be an elliptic curve with everywhere good reduction, and admitting complex multiplication by the ring of integers of a quadratic imaginary number field $K$. Then, for any $\lambda \in \mathcal{O}_K$ that is coprime to the number of roots of unity in $K$,
$$
A(L)_{\tors}\subseteq \ker(\Psi_L).
$$
\end{conjecture}

This conjecture has been proved by Srivastav and Taylor \cite{SrT}, at least when $\lambda$ has absolute norm coprime to $6$ (see also \cite{cnj}).  It is natural to ask what happens when one replaces the elliptic curve $A$ by an abelian variety of higher dimension.

The difficulties encountered when dealing with abelian varieties of arbitrary dimension can be summarized  as follows: (1) the theory of abelian functions in at least two variables is not quite the ``formula-fest'' that the theory of elliptic function is;  as a consequence, the construction of explicit bases in terms of values of these functions, which was at the center of the proof of the theorem of \cite{SrT}, is not possible anymore. (2) The geometric result proved by Pappas \cite[Theorem~A]{pappas98}, from which he was able to derive a new proof of the theorem of \cite{SrT}, turns out to be wrong in higher dimension \cite[Theorem~C]{pappas98}.

These observations, together with the work of the second author (see \cite{gil4}, Theorem~4.7 and Example~4.9), lead us to think that an analogue of the theorem of \cite{SrT}, in full generality, is very unlikely. Nevertheless, one may hope for the existence of abelian varieties for which Conjecture~\ref{martin_conj} holds.
 
Moreover, we observe that the conjecture says nothing about the behaviour of $\Psi_L$ on points of infinite order. It is expected that such points should not, in general, belong to the kernel of $\Psi_L$, but very few examples have been given so far.
 
In the present note, we consider these two questions. More precisely, we provide a family of examples of $(A,\lambda)$, where $A$ has complex multiplication by a cyclotomic field, such that:
\begin{itemize}
\item[(a)] Torsion points lie in the kernel of $\Phi$ (a variant of $\Psi$, defined below);
\item[(b)] There exists infinitely many independent points of infinite order, defined over quadratic extensions of the base field, which do not belong to the kernel of $\Psi$.
\end{itemize}

The point (a) is obtained as a consequence of the deep results of Greenberg \cite{greenberg81} on torsion points on Jacobians of quotients of Fermat curves. The point (b) follows from results of Levin and the second author \cite{GillibertLevin} on line bundles on hyperelliptic curves.

%%%%%%%%%%%%%%%%%%%%%%%%%%%%%%%%%%%%%%%%%%%%%
\subsection{Description of the results}
\label{subsection_results}

Once and for all, we fix a prime number $p\geq 5$, and a primitive $p$-th root of unity $\zeta$. We denote by $K=\Q(\zeta)$ the $p$-th cyclotomic field, by $K^+$ its maximal real subfield, by $h$ and $h^+$ their respective class numbers. We consider the smooth projective curve $C$ birational to the plane curve defined by the equation $y^p=x(1-x)$, and we  denote by $A$ the Jacobian of $C$. This abelian variety is defined over $\Q$. Over $K$ it has complex multiplication by $\Z[\zeta]$. We will identify in this note the algebraic integer $\lambda= 1-\zeta$ with the endomorphism of $A$ it induces. The group $G$ of points of $A(\overline {\Q})$ killed by $\lambda$ has order $p$. In fact, $G$ and its Cartier dual $G^D$ are constant group schemes over $K$.

We let $F$ be either the field $K$, if $2^{p-1}\equiv 1 \mod p^2$, or the field $K(2^{1/p})$ otherwise. We prove in Subsection~\ref{jacobian_red} that $A$ acquires everywhere good reduction over $F$.

Let us first explain how one defines the homomorphism $\Phi_F$. The group schemes $G$ and $G^D$ are constant of order $p$. Let us denote by $(\Z/p\Z)_{\mathcal{O}_F}$ the constant group scheme of order $p$ over $\mathcal{O}_F$, then by choosing a generator of $G^D$ we get a morphism of group schemes $g:(\Z/p\Z)_{\mathcal{O}_F}\to \mathcal{G}^D$ which becomes an isomorphism over $\mathcal{O}_F[p^{-1}]$. By Cartier duality, $g$ induces a morphism of group schemes $g^D:\mathcal{G}\to\mu_p$.

The following diagram describes the situation
\begin{equation}
\label{diagram1}
\begin{CD}
@. @. H^1(\mathcal{O}_F,\mathcal{G}) @>\pi>> \Pic(\mathcal{G}^D) \\
@. @. @V(g^D)_*VV @VVg^*V \\
1 @>>> \mathcal{O}_F^{\times}/(\mathcal{O}_F^{\times})^p @>>> H^1(\mathcal{O}_F,\mu_p) @>\pi>> \Pic((\Z/p\Z)_{\mathcal{O}_F}) \\
\end{CD}
\end{equation}
By functoriality of Waterhouse's construction (see \cite[Subsection~2.2]{gil3}), the square is commutative. Moreover, according to \cite[Prop.~3.1]{gil3}, the lower line is exact.

The morphism
$$
\Phi_F:A(F)\longrightarrow \Pic((\Z/p\Z)_{\mathcal{O}_F})\simeq \Cl(\mathcal{O}_F)^p
$$
is defined by: $\Phi_F=g^*\circ\Psi_F$. We note that, as a scheme, $(\Z/p\Z)_{\mathcal{O}_F}$ is the disjoint union of $p$ copies of $\Spec(\mathcal{O}_F)$, hence its Picard group is isomorphic to the product of $p$ copies of $\Cl(\mathcal{O}_F)$.

Let us describe this construction from an algebraic point of view: the maximal order $\mathfrak{M}$ in $F[G]$ is completely split and thus is a product of $p$ copies of $\mathcal{O}_F$. The map $\Phi_F$ can be understood as the composition of $\Psi_F$ with the homomorphism $\Cl(\mathfrak{A})\to\Cl(\mathfrak{M})$ induced by scalar extension.

According to the exactness of the lower line of \eqref{diagram1}, $\Phi_F(P)=0$ if and only if $F(\lambda^{-1}P)$ is of the form $F(\alpha^{1/p})$ for some $\alpha\in \mathcal{O}_F^{\times}$. This result we refer to as Kummer's criterion, was first proved by Bouklou \cite{bouklou}. Using this characterisation of $\ker(\Phi_F)$, we prove in Section~\ref{section_torsion} the following:

\begin{thm}
\label{theorem12}
For any integer $p\geq 5$ which doesn't divide $h^+$,
$$
A(F)_{\tors}\subseteq \ker(\Phi_F).
$$
\end{thm}

\begin{rmq}
\label{remark13}
\begin{enumerate}
\item Vandiver's conjecture states that $p$ never divides $h^+$. This conjecture has been checked for $p$ less than $163\times 10^6$ by Buhler and Harvey \cite{buha11}.
\item The group $H^1(\mathcal{O}_F,\mu_p)$ being killed by $p$, it follows from the construction of $\Phi_F$ that $p.\Phi_F(P)=0$ for all $P\in A(F)$. In particular, any torsion point with annihilator coprime to $\lambda$ belongs to the kernel of $\Phi_F$.  Moreover, the result given in Theorem~\ref{theorem12} is trivial when $p$ doesn't divide $h_F$ (the class number of $F$), therefore it is important to check that this is not the case in general. Since either $F=K$ or $F/K$ is totally ramified over $\mathfrak{p}$ (the unique prime of $K$ above $p$), then $h$ divides $h_F$ (see \cite[Theorem~10.1]{washington}). So we are interested in primes $p$ that divide $h$. Such primes are called irregular, and it is well known that there exists infinitely many of them.
\item In fact, the group schemes $\mathcal{G}$ and $\mathcal{G}^D$ are isomorphic over $\mathcal{O}_F$ (see Subsection~\ref{theorem14_proof}). It would be interesting to compute the cokernel of the map $H^1(\mathcal{O}_F,\mathcal{G})\to H^1(\mathcal{O}_F,\mu_p)$, which would allow to compare the kernels of $\Phi_F$ and $\Psi_F$.
\end{enumerate}
\end{rmq}

We now state our result about points that do not belong to the kernel of $\Psi$. We note that the curve $C$ with equation $y^p=x(1-x)$ has good reduction outside $p$, hence its jacobian $A$ also has good reduction outside $p$. If $L$ is a finite extension of $K$, we let $S$ be the set of places of $L$ above $p$. We prove in Section~\ref{section_infinite} the following:

\begin{thm}
\label{theorem14}
There exists an infinity of quadratic extensions $L/K$ with a point $P\in A(L)$ of infinite order such that $\Psi_{L,S}(P)\neq 0$.
\end{thm}

In fact, the result we prove (Corollary~\ref{class_invariant_cor}) applies in a quite general setting to Jacobians of hyperelliptic curves defined over $\Q$.

As we already pointed out, $\mathcal{G}^D$ becomes isomorphic to the constant group scheme $(\Z/p\Z)$ after inverting $p$, hence the corresponding Picard group is isomorphic to the product of $p$ copies of $\Cl(\mathcal{O}_{L,S})$. Roughly speaking, the equality $\Psi_{L,S}=\Phi_{L,S}$ holds.

\subsubsection*{Acknowledgements}
The authors thank Qing Liu for his help, and Jes{\'u}s G{\'o}mez Ayala for pointing out Greenberg's result.

%%%%%%%%%%%%%%%%%%%%%%%%%%%%%%%%%%%%%%%%%%%%%
%%%%%%%%%%%%%%%%%%%%%%%%%%%%%%%%%%%%%%%%%%%%%

\section{Torsion points and trivial Galois modules}
\label{section_torsion}

%%%%%%%%%%%%%%%%%%%%%%%%%%%%%%%%%%%%%%%%%%%%%
\subsection{Ramification in the splitting field of $X^p-l$ }

Let $l$ and $p$ be distinct prime numbers and let $N$ denote the splitting field of the polynomial $X^p-l$ in $\overline{\Q}$. The extension $N/\Q$ is unramified outside $p$ and $l$. Moreover the ramification over $l$ is easy to describe. In  this section we  study in more details  the ramification of the prime $p$ in $N$. The results of Proposition~\ref{prop23} will be used in order to construct a  number field on which   the abelian varieties we  deal with have everywhere good reduction. 

We begin by recalling  Hensel's lemma. 

\begin{lem}
\label{lem21}
Let $f \in \Z_p[X]$ and suppose that there exists $x \in \Z_p$ such that $v_p(f'(x))=k$ and $v_p(f(x))\geq 2k+1$. Then there exists $y \in \Z_p$ with 
$v_p(y-x)\geq k+1$ such that $f(y)=0$.
\end{lem}

\begin{proof}
See for instance \cite[3.25]{frotay}.
\end{proof}

\begin{cor}
The polynomial $X^p-l$ has a root in $\Z_p$ iff $l^{p-1}\equiv 1\mod p^2$. 
\end{cor}

\begin{proof}
Suppose that $x$ is a root of $X^p-l$ in $\Z_p$. Since $x\equiv x^p\equiv l\mod p$, we may 
write $x=l+ap$ with $a \in \Z_p$. By using that $(l+ap)^p\equiv l\mod p^2$, we deduce that $l^p\equiv l\mod  p^2$ and so that $l^{p-1}\equiv 1\mod  p^2$.

Assume now that $l^{p-1} \equiv 1\mod p^2$. We choose $a\in \Z_p$ such that $l-l^p=ap^2$ and we set $x=l+ap$. Since $x^p\equiv l\mod p$, we observe that $v_p(f'(x))=1$. Moreover we have the congruences 
$x^p=(l+ap)^p\equiv l^p+ap^2 \equiv l\mod p^3$ and thus $v_p(f(x))\geq 3$. It follows now from Lemma~\ref{lem21} that $f$ has a root in $\Z_p$.
\end{proof}

Let $\zeta$ be a primitive $p$-th root of unity and let $K=\Q(\zeta)$ be the $p$-th cyclotomic field. Let $l^{1/p}$  be  a root in $\overline {\Q}$ of $X^p-l$ and let $L$ denote the field $\Q(l^{1/p})$.   The compositum 
of $L$ and $K$, denoted by $N$, is the splitting field of $X^p-l$. The extension $N/\Q$ is of degree $p(p-1)$.  
 
\begin{prop}
\label{prop23}
Let $\mathfrak{p}$ be the unique prime ideal of $K$ above $p$. Then $\mathfrak{p}$ is unramified in $N$ iff  $l^{p-1}\equiv 1\mod p^2$. Moreover in this case $p\mathcal{O}_L=\mathfrak{s}\mathfrak{t}^{p-1}$ where $\mathfrak{s}$ and $\mathfrak{t}$ are distinct prime ideals of $\mathcal{O}_L$.
\end{prop}

\begin{proof}
We let $K_{\mathfrak{p}}$ be the completion of $K$ at $\mathfrak{p}$ and we set 
$$
N_{\mathfrak{p}}=K_{\mathfrak{p}}\otimes_K N\simeq K_{\mathfrak{p}}[X]/(x^p-l).
$$
Firstly, we easily check that since $K_{\mathfrak{p}}$ contains the $p$-th roots of unity, then $X^p-l$ is either completely split or irreducible over $K_{\mathfrak{p}}$. 

If $l^{p-1}\equiv 1\mod p^2$, the polynomial $X^p-l$ has a root in $\Q_p$ and so is completely split in $K_{\mathfrak{p}}$. The $K_{\mathfrak{p}}$-algebra $N_{\mathfrak{p}}$ is the 
product of $p$-copies of $K_{\mathfrak{p}}$ and thus $\mathfrak{p}$ is completely split in $N$.

Now if we suppose that $l^{p-1}\not\equiv 1\mod p^2$, then $N_{\mathfrak{p}}/K_{\mathfrak{p}}$ is a field extension of degree $p$. If $\mathfrak{p}$ is inert in $N$ the extension 
$N_{\mathfrak{p}}/K_{\mathfrak{p}}$ is unramified.  Therefore there exists a root of unity 
$\mu$ such that $N_{\mathfrak{p}}=K_{\mathfrak{p}}(\mu)=\Q_p(\zeta, \mu)$. We deduce that $N_{\mathfrak{p}}/\Q_p$ is an abelian extension. Since the Galois groups of $N_{\mathfrak{p}}/\Q_p$ and $N/\Q$ are isomorphic in this case,  this is impossible. We conclude that  $\mathfrak{p}$ is ramified in $N$. 

We now assume that $l^{p-1} \equiv 1\mod p^2$. We consider 
$$L_p=\Q_p\otimes_\Q L\simeq \Q_p[X]/(X^p-l).$$
We denote by $x$ a root of $X^p-l$ in $\Q_p$ and we write $X^p-l=(X-x)f(X)$ with $f(X)\in \Q_p[X]$. If $z$ is a root of $f(X)$ in an algebraic closure of 
$\Q_p$ then $\Q_p(z)=\Q_p(\zeta)$. This implies  that $f(X)$ is irreducible over $\Q_p$ and that
$$
L_p\simeq \Q_p\times \Q_p(\zeta).
$$
The result follows at once from this decomposition. 
\end{proof}

\begin{rmq}
If $l^{p-1}\not\equiv1$ mod $p^2$, then $p$ is totally ramified in $N$. If $l^{p-1}\equiv1$ mod $p^2$, then there are exactly $p$ prime ideals of $N$ above $p$, that we denote by $\mathcal{P}_1,\mathcal{P}_2,\dots,\mathcal{P}_p$. One may observe that the ramification groups $I_1, I_2,\dots,I_p$ of these ideals are precisely the $p$ distinct subgroups of $\Gal(N/\Q)$ of order $p-1$. For $i=1,\dots,p$, the field $L_i=N^{I_i}$ is generated by a root of $X^p-l$. It follows from Proposition~\ref{prop23} that $p\mathcal{O}_{L_i}= \mathfrak{s}_i\mathfrak{t}_i^{p-1}$ with $\mathfrak{s}_i\mathcal{O}_N=\mathcal{P}_i^{p-1}$ and 
$\mathfrak{t}_i\mathcal{O}_N=\prod_{j\neq i}\mathcal{P}_j$.
\end{rmq}

\begin{rmq}
A prime number $p$ such that $2^{p-1}\equiv 1\mod p^2$ is called a Wieferich prime. The only known Wieferich primes are $1093$ and $3511$. However, J.~Silverman has proved that, if the \emph{abc} conjecture holds, then for every integer $n$ there exists an infinite number of primes $p$ such that $n^{p-1}\equiv 1\mod p^2$. Hence there should exist infinitely many Wieferich primes.
\end{rmq}

%%%%%%%%%%%%%%%%%%%%%%%%%%%%%%%%%%%%%%%%%%%%%
\subsection{Reduction of the Jacobian of $C$}
\label{jacobian_red}

Let $L$ be a number field. Let us recall that a hyperelliptic curve of genus $g\geq 1$ over $L$ has a rational Weierstrass point if and only if it is birational to a plane curve defined by an equation of the form
$$
y^2=f(x)
$$
where $f\in L[X]$ is a square-free polynomial of degree $2g+1$.

Such curves are sometimes called imaginary hyperelliptic curves. They have a unique point at infinity, that we denote by $P_{\infty}$. We briefly recall the well-known description of $2$-torsion points on the Jacobian of such a curve.

\begin{lem}
\label{archi_connu}
Let $C$ be a smooth hyperelliptic curve of genus $g$ over $\overline{\Q}$, birational to the plane curve defined by the equation
$$
y^2=f(x)
$$
with $f\in \overline{\Q}[X]$ square-free of degree $2g+1$. Let $t_1,t_2,\dots,t_{2g+1}$ be the roots of $f$ over $\overline{\Q}$. For all $i$ we let $P_i:=(t_i,0)$ be the corresponding point on $C$, and $D_i=(P_i-P_{\infty})$ be the divisor class of $P_i-P_{\infty}$. Then $D_1,\dots,D_{2g+1}$ span $\Pic^0(C)[2]$ as a $\mathbb{F}_2$-vector space, and are subject to the single relation
$$
\sum_{i=1}^{2g+1} D_i=0.
$$
\end{lem}

\begin{proof}
Let $x$ and $y$ be the rational functions on $C$ given by the coordinates. One checks that
$$
\divisor(x-t_i)=2P_i-2P_{\infty}\qquad \text{for $i=1,\dots,2g+1$}
$$
hence the $D_i$ have order $2$. Moreover, since
$$
\divisor(y)=\left(\sum_{i=1}^{2g+1} P_i\right)-(2g+1)P_{\infty}
$$
the sum of all $D_i$ is zero. It follows from the Riemann-Roch theorem that there is no other relation between the $D_i$. Therefore, the $D_i$ span a $\mathbb{F}_2$-vector space of dimension $2g$. Since the order of the full group $\Pic^0(C)[2]$ is $2^{2g}$, it follows that the $D_i$ span $\Pic^0(C)[2]$.
\end{proof}

We briefly recall our notations from Subsection~\ref{subsection_results}: $p\geq 5$ is a prime number, $\zeta$ is a primitive $p$-th root of unity, and $K=\Q(\zeta)$ is the $p$-th cyclotomic field.

\begin{prop}
\label{prop_jac_red}
Let $C$ be the smooth projective curve over $\Q$ birational to the curve with equation
$$
y^p=x(1-x).
$$
Then:
\begin{enumerate}
\item[(1)] The curve $C$ is hyperelliptic of genus $(p-1)/2$ with a rational Weierstrass point.
\item[(2)] The Jacobian $A$ of $C$ has complex multiplication by $\Z[\zeta]$.
\item[(3)] The variety $A$ has everywhere good reduction over the field $K(2^{1/p})$. Moreover, it has everywhere good reduction over $K$ iff $2^{p-1}\equiv 1\mod p^2$.
\end{enumerate}

\end{prop}

\begin{proof}
(1) Putting $v=x-1/2$ and $u=-y$, the equation $y^p=x(1-x)$ is equivalent to the equation $v^2=u^p+1/4$, hence the result.

(2) We note that the group scheme $\mu_p$ acts on $C$ (by multiplying $y$ by a $p$-th root of unity), and this action induces an action of $\mu_p$ on $A$. Hence, the abelian variety $A$ has complex multiplication by $\Z[\zeta]$. In particular, it has potential good reduction.

(3) For any integer $m$ we let $K(A[m])$ be the smallest extension of $K$ over which the elements of $A[m]$ are rational. By virtue of Lemma~\ref{archi_connu}, the field $K(A[2])$ is the splitting field of the polynomial $u^p+1/4$ over $K$, hence $K(A[2])=K(2^{1/p})$. Next we observe that since $A/K$ has  potential good reduction,  we may use \cite[Corollary~2]{st68} in order to  study its reduction.  
From \cite[Theorem~4]{greenberg81} we know that $K(A[p])$ is a subextension of the field $K(\{\varepsilon^{1/p}\ |\ \varepsilon \in U\})$, where $U$ denotes the group of cyclotomic units in $K^+$ (the maximal real subfield of $K$). We therefore deduce  that $K(A[p])/K$ is unramified outside $\mathfrak{p}$, the unique prime of $K$ above $p$.  By \cite[Corollary~2~b)]{st68}, we conclude that $A/K$ has good reduction outside $\mathfrak{p}$, everywhere good reduction over $K(2^{1/p})$, and everywhere good reduction over $K$ if and only if $\mathfrak{p}$ is unramified in  $K(2^{1/p})/K$. The result then follows from Proposition~\ref{prop23}.
\end{proof}

%%%%%%%%%%%%%%%%%%%%%%%%%%%%%%%%%%%%%%%%%%%%%
\subsection{Proof of Theorem~\ref{theorem12}}

We now use the notations of Subsection~\ref{subsection_results}.

We let $F$ be either the field $K$, if $2^{p-1}\equiv 1\mod p^2$, or  the field $K(2^{1/p})$ otherwise. The torsion subgroup of  $A(K)$ has been studied by various authors (see \cite{gr78}, \cite{greenberg81} and \cite{tzermias}). In particular we know from \cite[Theorem 1]{tzermias} that
$$
A(K)_{\tors}=A[\lambda^3]
$$
with $\lambda=(1-\zeta)$. We quickly check that this implies that $A[p^{\infty}](F)=A[\lambda^3]$. Indeed, the result is true if $F=K$. Now suppose that $F=K(2^{1/p})$. Since $F/K$ is of degree $p$, if there were a $\lambda^4$-torsion (but not $\lambda^3$-torsion) point $R$ rational over $F$, then we would have $F=K(R)$ and $F$ would be a subfield of $K(A[p])$. But this is impossible since $F/K$ is ramified at every prime ideal above $2$ while $K(A[p])$ is unramified outside $p$.

We let $P$ be a primitive $\lambda^3$-torsion point, and we fix some $Q \in A(\overline{\Q})$ such that $\lambda Q=P$. Then the extension $F(\lambda^{-1}P)/F$ is nothing else than $F(Q)/F$. According to Kummer's criterion, in order to prove the theorem it suffices to show that $F(Q)/F$ has a Kummer generator which is a unit. This is the goal of the next proposition, which can be easily deduced from \cite[Theorem~4]{greenberg81}.

\begin{prop}
\label{prop41}
There exists a cyclotomic unit $\mu$ of $K$ such that $F(Q)=F(\mu^{1/p})$.
\end{prop}

\begin{proof}
The field $F(Q)$ is the compositum of $F$ and $K(Q)$, which are linearly disjoint over $K$. Moreover, $K(Q)/K$ is a Kummer subextension of degree $p$ of $K(A[p])$.  We let $K^+$ be the maximal real subfield of $K$. From \cite[Theorem 4]{greenberg81}  we know that $K(A[p])$ is a subfield  of $M=K(\{\varepsilon^{1/p}\ |\ \varepsilon \in U\})$, where $U$ denote the group of cyclotomic units of $K^+$. Therefore, in order to prove the proposition,  it is enough to show that  every subextension of  
$M$ of degree $p$ over $K$ is generated by a $p$-th root of a cyclotomic unit. We shall prove the claim by showing that,  in the field extension $M/K$, the number $n_p$ of subextensions  of degree $p$, is equal to the number $n'_p$ of distinct  subextensions generated by a $p$-th root of a cyclotomic unit. We know from Galois theory that $n_p$ is equal to the number of subgroups of $\mathbb{F}_p^t$ of index $p$, where $t$ denotes the dimension of the $\mathbb{F}_p$-vector space $U/U^p$. Hence we conclude that $n_p=\frac{p^t-1}{p-1}$. 

Now let us compute $n'_p$.

\begin{lem}
\label{lem42}
We assume that $p$ doesn't divide the class number of $K^+$. Let $\alpha$ and $\beta$ be cyclotomic units. Then $K(\alpha^{1/p})=K(\beta^{1/p})$ iff there exists a cyclotomic unit $\gamma$ and an integer $k$ coprime to $p$, such that $\beta=\gamma^p\alpha^k$. 
\end{lem}

\begin{proof}
Suppose that $K(\alpha^{1/p})=K(\beta^{1/p})$.  From  Kummer theory  we deduce  that  there exists  $\delta \in K$  and an integer $k$ coprime to $p$, such that $\beta=\delta^p\alpha^k$. Since $\alpha$ and $\beta$ are units, then $\delta$ is a unit of $K$ and so  can be written as a product $\delta=\zeta \gamma$ where $\zeta$ is a $p$-th root 
of unity and $\gamma$ a unit in $K^+$. Therefore $\beta=\gamma^p\alpha^k$ and thus  $\gamma^p$ is a cyclotomic unit of $K^+$. Since under our hypotheses the index of $U$ in the full group of units of $K^+$ is coprime to  $p$,  we conclude that $\gamma$ belongs to $U$ as claimed.
\end{proof}

It follows from Lemma~\ref{lem42} that there exists a bijection between the set of subextensions of $M/K$ generated by a $p$-th root of a cyclotomic unit, and the set of lines of the $\mathbb{F}_p$-vector space $U/U^p$. Therefore we obtain that $n'_p=\frac{p^t-1}{p-1}=n_p$, which completes the proof of  Proposition~\ref{prop41} and of Theorem~\ref{theorem12}.
\end{proof}

%%%%%%%%%%%%%%%%%%%%%%%%%%%%%%%%%%%%%%%%%%%%%
%%%%%%%%%%%%%%%%%%%%%%%%%%%%%%%%%%%%%%%%%%%%%

\section{Points of infinite order and non-trivial Galois modules}
\label{section_infinite}

In this section, we first explain how to derive, from the work of Levin and the second author \cite{GillibertLevin}, non-vanishing results for the class-invariant homomorphism on Jacobians of hyperelliptic curves. Then we focus on the family of Jacobians previously considered.

%%%%%%%%%%%%%%%%%%%%%%%%%%%%%%%%%%%%%%%%%%%%%
\subsection{Class-invariants for Jacobians}
\label{subsection31}

Let $C$ be a smooth, geometrically connected, projective curve over $\Q$, and let $S$ be the set of primes of bad reduction of $C$. If $L$ is a number field, we denote by $\mathcal{O}_{L,S}$ the set of $S_L$-integers of $L$, where $S_L$ is the set of places of $L$ above $S$. In fact, $\mathcal{O}_{L,S}=\mathcal{O}_L[S^{-1}]$.

Let $\C\to \Spec(\Z_S)$ be the smooth projective model of $C$. If $P\in C(L)$ is a point, then by abuse of notation we also denote by $P:\Spec(\mathcal{O}_{L,S})\to\C$ the associated section of the integral model of $C$.

\begin{prop}
\label{super}
Assume $C$ is a hyperelliptic curve over $\Q$ with a rational Weierstrass point, and let $m>1$ be an integer. Then there exists an infinity of imaginary quadratic fields $L/\Q$ with a point $P\in C(L)$ such that the specialisation map
\begin{align*}
P^*:\Pic^0(\C)[m]\;\longrightarrow & \;\Cl(\mathcal{O}_{L,S}) \\
\mathcal{L}\;\longmapsto & P^*\mathcal{L}
\end{align*}
is injective.
\end{prop}

\begin{proof}
See \cite[Cor.~3.2]{GillibertLevin}.
\end{proof}

Moreover, the infinity in the statement above can be made explicit: if we let $g$ be the genus of $C$, then the number of imaginary quadratic fields $L$ of discriminant $d_L$ such that $|d_L|<X$ is $\gg X^{\frac{1}{2g+1}}/\log X$.

We now want to underline the consequences of this result for the study of the class-invariant homomorphism.

Let us assume that the hypotheses of Proposition~\ref{super} above hold. Let $\J\rightarrow \Spec(\Z_S)$ be the N\'eron model of $\Jac(C)$, which is an abelian scheme over $\Spec(\Z_S)$. Because $\C\to\Spec(\Z_S)$ is a smooth relative curve, with geometrically connected fibers, its relative Picard functor is representable by a separated scheme, and, according to \cite[\S{}9.5, Thm.~1]{NeronModels}, the connected component of its Picard functor is isomorphic to $\J$. Let us now fix a point $\varepsilon\in C(\Q)$. Then, according to \cite[\S{}8.2, Prop.~4]{NeronModels}, we can associate to the corresponding section $\varepsilon:\Spec(\Z_S)\to\C$ a universal line bundle, namely, the Poincar{\'e} line bundle, that we denote by $\mathcal{P}_{\C}$, on $\C\times \J$. It follows from the universal property of $\mathcal{P}_{\C}$ that, given a point $P\in C(L)$, the map
$$
P^*:\Pic^0(\C)\longrightarrow \Cl(\mathcal{O}_{L,S})
$$
can be identified with the map
\begin{align*}
\Jac(C)(\Q)\;\longrightarrow & \;\Cl(\mathcal{O}_{L,S}) \\
Q\;\longmapsto & \; (P\times Q)^*\mathcal{P}_{\C}
\end{align*}

We now consider the embedding $C\rightarrow \Jac(C)$ associated with $\varepsilon$. Using the  smoothness of $\C$ and the universal property of $\J$, this embedding extends into a closed immersion $i:\C\rightarrow \J$. This immersion gives rise to an isomorphism between $\J$ and its dual abelian scheme, encoded again into a Poincar{\'e} line bundle, that we denote by $\mathcal{P}_{\J}$, on $\J\times\J$. By construction, we have that $\mathcal{P}_{\C}=(i\times \mathrm{id})^*\mathcal{P}_{\J}$, hence
$$
(P\times Q)^*\mathcal{P}_{\C}=(i(P)\times Q)^*\mathcal{P}_{\J}
$$

Finally, according to the Manin-Mumford conjecture (proved by Raynaud), only finitely many $i(P)$ are torsion points. Therefore, we deduce from Proposition~\ref{super} the following:

\begin{cor}
\label{jacobian_cor}
Assume $C$ is a hyperelliptic curve over $\Q$ with a rational Weierstrass point, and let $m>1$ be an integer. Then there exists an infinity of imaginary quadratic fields $L/\Q$ with a point $P\in \Jac(C)(L)$ of infinite order, such that the map
\begin{align*}
\Jac(C)(\Q)[m]\;\longrightarrow & \;\Cl(\mathcal{O}_{L,S}) \\
Q\;\longmapsto & \;(P\times Q)^*\mathcal{P}_{\J}
\end{align*}
is injective.
\end{cor}

Let us translate Corollary~\ref{jacobian_cor} in terms of the class-invariant homomorphism. We now assume that $\Jac(C)$ has a $\Q$-rational point $Q$ of prime order $p\neq 2$. We shall apply Corollary~\ref{jacobian_cor} with $m=p$. The point $Q$ defines an embedding $(\Z/p\Z)_{\Q}\to \Jac(C)$ where $(\Z/p\Z)_{\Q}$ denotes the constant group scheme over $\Spec(\Q)$. Because $(\Z/p\Z)_{\Z_S}$ is smooth, this embedding extends into a morphism $f:(\Z/p\Z)_{\Z_S}\to \J$. The image of $f$ is isomorphic to $(\Z/p\Z)_{\Z_S}$ (in other terms, $f$ is an embedding). Indeed, $(\Z/p\Z)_{\Z_S}$ is the unique finite flat group scheme over $\Spec(\Z_S)$ extending $(\Z/p\Z)_{\Q}$.

Hence we have an exact sequence (for the fppf topology on $\Spec(\Z_S)$)
$$
\begin{CD}
0 @>>> (\Z/p\Z)_{\Z_S} @>f>> \J @>\varphi>> \mathcal{B} @>>> 0\\
\end{CD}
$$
where $\mathcal{B}$ is an abelian scheme over $\Spec(\Z_S)$. By duality of abelian schemes, we have a dual exact sequence
\begin{equation}
\label{dualsequence}
\begin{CD}
0 @>>> \mu_p @>>> \mathcal{B}^t @>\varphi^t>> \J @>>> 0\\
\end{CD}
\end{equation}
where $\mathcal{B}^t$ is the dual abelian scheme of $\mathcal{B}$ (remember that $\J$ is self-dual), $\varphi^t$ is the dual isogeny of $\varphi$, and $\mu_p$ is the group scheme of roots of unity.

If $L$ is a number field, we define $\psi_{L,S}$ as the composition of the maps
$$
\begin{CD}
\Jac(C)(L)=\J(\mathcal{O}_{L,S}) @>\delta>> H^1(\mathcal{O}_{L,S},\mu_p) @>\pi>> \Pic((\Z/p\Z)_{\mathcal{O}_{L,S}})\simeq \Cl(\mathcal{O}_{L,S})^p \\
\end{CD}
$$
where $\delta$ is the coboundary map deduced from the exact sequence \eqref{dualsequence} by considering $\Spec(\mathcal{O}_{L,S})$-sections, and $\pi$ is Waterhouse's morphism. In fact, $\psi_{L,S}$ is nothing else that the class-invariant homomorphism associated to the isogeny $\varphi^t$.

\begin{cor}
\label{class_invariant_cor}
There exists an infinity of imaginary quadratic fields $L/\Q$ with a point $P\in \Jac(C)(L)$ of infinite order such that $\psi_{L,S}(P)\neq 0$.
\end{cor}

\begin{proof}
It is well known (see for example \cite[Prop.~3.1]{gil3}) that, given $t\in H^1(\mathcal{O}_{L,S},\mu_p)$,
$$
\pi(t)=(0,c(t),c(t)^2,\dots c(t)^{p-1})
$$
where $c$ is the natural map $H^1(\mathcal{O}_{L,S},\mu_p)\to \Cl(\mathcal{O}_{L,S})$ given by Kummer theory. Moreover, it follows from the arguments of \cite[Lemme~3.2]{gil1} that
$$
c(\delta(P))=(P\times Q)^*\mathcal{P}_{\J}
$$
hence the result, by virtue of Corollary~\ref{jacobian_cor}.
\end{proof}

\begin{rmq}
When enlarging $S$, one enlarges the kernel of $\psi_{L,S}$. Therefore, when proving triviality results, one should avoid to invert places, and conversely one may invert places while proving non triviality results.
\end{rmq}

%%%%%%%%%%%%%%%%%%%%%%%%%%%%%%%%%%%%%%%%%%%%%
\subsection{Proof of Theorem~\ref{theorem14}}
\label{theorem14_proof}

We come back to the situation considered in Subsection~\ref{subsection_results}. In particular, $C$ now denotes the smooth projective curve over $\Q$ birational to the curve with equation $y^p=x(1-x)$, which, according to Proposition~\ref{prop_jac_red}, is hyperelliptic with a rational Weierstrass point, and $A$ denotes the Jacobian of $C$.

Using the Jacobian criterion, one checks that the set of primes of bad reduction of $C$ is $S=\{p\}$. If we let $P_0=(0,0)$ and $P_{\infty}$ be the point at infinity, then the divisor of the $x$-function on $C$ is $pP_0-pP_{\infty}$, hence the class of the divisor $P_0-P_{\infty}$ is a point of order $p$ in $A(\Q)$, that we denote by $Q$. Moreover, $A$ has complex multiplication by $\Z[\zeta]$, and $Q$ is invariant under the action of $\zeta$. Hence, over the field $K$, the subgroup generated by $Q$ coincides with $A[\lambda]$, where $\lambda:=1-\zeta$.

Over $\Q$, we have an isogeny $\varphi:A\to B$ with cyclic kernel of order $p$, generated by $Q$. If we base change this isogeny to $K$, we recover the isogeny $\lambda:A\to A$ (because the two isogenies have the same kernel). This implies that $B$ is isomorphic to $A$ over $K$. However, the endomorphism $\lambda$ being not defined over $\Q$, the variety $B$ is not isomorphic to $A$ over $\Q$.

Let $L$ be a number field containing $K$. On the one hand, we may consider the homomorphism $\psi_{L,S}$ defined in Subsection~\ref{subsection31}, which is obtained by dividing points by the isogeny $\varphi^t$. On the other hand, the homomorphism $\Psi_{L,S}$ from the introduction is obtained by dividing points by the isogeny $\lambda$. Our aim is to prove that these two maps $\psi_{L,S}$ and $\Psi_{L,S}$ coincide, up to some automorphism of $A$. As we said above, the equality $\lambda=\varphi$ holds over $L$. Therefore, it suffices to prove that the isogeny $\lambda$ is, up to some automorphism of $A$, its own dual.

Indeed, let $\lambda^t$ be the dual of $\lambda$, then $\lambda^t$ is an endomorphism of $A$, hence corresponds to an element of $\Z[\zeta]$. Let $\mathfrak{p}$ be the unique prime of $K$ above $p$, then it is well known that $\mathfrak{p}^{p-1}=(p)$ and that $\lambda$ is a generator of $\mathfrak{p}$. It follows that $\lambda^{p-1}$ is (up to some automorphism of $A$) the multiplication by $p$ on $A$, that we denote by $[p]$. Hence, $(\lambda^{p-1})^t=(\lambda^t)^{p-1}$ is the dual of $[p]$, than one identifies with $[p]$ itself by auto-duality of $A$. Therefore, the ideal generated by $(\lambda^t)^{p-1}$ is equal to $(p)$, which implies that $\lambda^t=\lambda$ up to a unit in $\Z[\zeta]$. Hence the result.

Finally, Theorem~\ref{theorem14} follows from Corollary~\ref{class_invariant_cor} and the following remark: when performing the base change $K/\Q$, which has degree coprime to $p$, we preserve the $p$-torsion of the class group of $\mathcal{O}_{L,S}$.

\begin{rmq}
The identification between $\lambda:A\to A$ and its dual can be extended at the level of N\'eron models. This has the following consequence: the group schemes $\mathcal{G}$ and $\mathcal{G}^D$ are isomorphic (over $\mathcal{O}_F$), in other terms $\mathcal{G}$ is self-dual. In particular, the special fiber of $\mathcal{G}$ at the unique prime of $F$ above $p$ has to be $\alpha_p$, the unique self-dual group scheme of order $p$ over $\mathbb{F}_p$.
\end{rmq}

%%%%%%%%%%%%%%%%%%%%%%%%%%%%%%%%%%%%%%%%%%%%%
%%%%%%%%%%%%%%%%%%%%%%%%%%%%%%%%%%%%%%%%%%%%%

\vskip 20pt

Institut de Math\'ematiques de Bordeaux

CNRS UMR 5251

351, cours de la Lib\'eration

F-33400 Talence, France.
\medskip

\texttt{Philippe.Cassou-Nogues@math.u-bordeaux1.fr}

\texttt{Jean.Gillibert@math.u-bordeaux1.fr}

\texttt{Arnaud.Jehanne@math.u-bordeaux1.fr}

\end{document}